\documentclass{amsart}
\usepackage{amsfonts}

\newtheorem{theorem}{Theorem}
\theoremstyle{plain}

\newtheorem{corollary}{Corollary}

\newtheorem{definition}{Definition}
\newtheorem{example}{Example}

\newtheorem{proposition}{Proposition}
\newtheorem{remark}{Remark}

\numberwithin{equation}{section}
\input{tcilatex}

\begin{document}
\title[$MN$-convex functions]{ On Weighted Means and $MN$-convex functions}
\author{\.{I}mdat \.{I}\c{s}can}
\address{Department of Mathematics, Faculty of Arts and Sciences,\\
Giresun University, 28200, Giresun, Turkey.}
\email{imdat.iscan@giresun.edu.tr}

\begin{abstract}
In this paper, we give more general definitions of weighted means and $MN$%
-convex functions. Using these definitions, we also obtain some generalized
results related to properties of $MN$-convex functions. The importance of
this study is that the results of this paper can be reduced to different
convexity classes by considering the special cases of $M$ and $N$.
\end{abstract}

\maketitle

\section{Introduction}

\bigskip The notions of convexity and concavity of a real-valued function of
a real variable are well known \cite{RV73}. The generalized condition of
convexity, i.e. $MN$-convexity with respect to arbitrary means $M$ and $N$,
was proposed in 1933 by Aumann \cite{A33}. Recently many authors have dealt
with these generalizations. In particular, Niculescu \cite{N03} compared $MN$%
-convexity with relative convexity. Andersen et al. \cite{AVV07} examined
inequalities implied by $MN$-convexity. In \cite{AVV07}, Anderson et al.
studied certain generalizations of these notions for a positive-valued
function of a positive variable as follows:

\begin{definition}
\bigskip A function $M:(0,\infty )\times (0,\infty )\rightarrow (0,\infty )$
is called a Mean function if

\begin{enumerate}
\item[(M1)] $M(u,v)=M(v,u),$

\item[(M2)] $M(u,u)=u,$

\item[(M3)] $u<M(u,v)<v$ whenever $u<v,$

\item[(M4)] $M(\lambda u,\lambda v)=\lambda M(u,v)$ for all $\lambda >0.$
\end{enumerate}
\end{definition}

\begin{example}
For $u,v\in (0,\infty )$ 
\begin{equation*}
M(u,v)=A(u,v)=A=\frac{u+v}{2}
\end{equation*}
is the Arithmetic Mean, 
\begin{equation*}
M(u,v)=G(u,v)=G=\sqrt{uv}
\end{equation*}
is the Geometic Mean, 
\begin{equation*}
M(u,v)=H(u,v)=H=A^{-1}(u^{-1},v^{-1})=\frac{2uv}{u+v}\text{ }
\end{equation*}
is the Harmonic Mean, 
\begin{equation*}
M(u,v)=L(u,v)=L=\left\{ 
\begin{array}{cc}
\frac{u-v}{\ln u-\ln v} & u\neq v \\ 
u & u=v%
\end{array}
\right.
\end{equation*}
is the Logarithmic Mean, 
\begin{equation*}
M(u,v)=I(u,v)=I=\left\{ 
\begin{array}{cc}
\frac{1}{e}\left( \frac{u^{u}}{v^{v}}\right) ^{\frac{1}{u-v}} & u\neq v \\ 
u & u=v%
\end{array}
\right.
\end{equation*}
is the Identric Mean, 
\begin{equation*}
M(u,v)=M_{p}(u,v)=M_{p}=\left\{ 
\begin{array}{cc}
A^{1/p}(u^{p},v^{p})=\left( \frac{u^{p}+v^{p}}{2}\right) ^{1/p} & p\in 
\mathbb{R}
\backslash \left\{ 0\right\} \\ 
G(u,v)=\sqrt{uv} & p=0%
\end{array}
\right.
\end{equation*}
is the $p$-Power Mean, In particular, we have the following inequality 
\begin{equation*}
M_{-1}=H\leq M_{0}=G\leq L\leq I\leq A=M_{1}.
\end{equation*}
\end{example}

Anderson et al. in \cite{A47} developed a systematic study to the classical
theory of continuous and midconvex functions, by replacing a given mean
instead of the arithmetic mean.

\begin{definition}
Let $M$ and $N$ be two means defined on the intervals $I$ $\subset \left(
0,\infty \right) $ and $J\subset \left( 0,\infty \right) $ respectively, a
function $f:I\rightarrow J$ is called $MN$-midpoint convex if it satisfies 
\begin{equation*}
f\left( M(u,v)\right) \leq N\left( f(u),f(v)\right)
\end{equation*}%
for all $u,v\in I$.
\end{definition}

The concept of $MN$-convexity has been studied extensively in the literature
from various points of view (see e.g. \cite{A47,A33,M04,N03}),

Let $A\left( u,v,\lambda \right) =\lambda u+(1-\lambda )v$, $G\left(
u,v,\lambda \right) =u^{\lambda }v^{1-\lambda },$ $H\left( u,v,\lambda
\right) =uv/(\lambda u+(1-\lambda )v)$ and $M_{p}\left( u,v,\lambda \right)
=\left( \lambda u^{p}+(1-\lambda )v^{p}\right) ^{1/p}$ be the weighted
arithmetic, geometric, harmonic , power of order $p$ means of two positive
real numbers $u$ and $v$ with $u\neq v$ for $\lambda \in \left[ 0,1\right] ,$
respectively. $M_{p}\left( u,v,\lambda \right) $ is continuous and strictly
increasing with respect to $\lambda \in \mathbb{R}$ for fixed $p\in 
\mathbb{R}
\backslash \left\{ 0\right\} $ and $a,b>0$ with $a>b.$ See \cite%
{I13d,I13,M04,M19,N00,N03} for some kinds of convexity obtained by using
weighted means.

The aims of this paper, a general definition of weighted means and a general
definition of $MN$-convex functions via the weighted means is to give. In
recent years, many studies have been done by considering the special cases
of $M$ and $N$. The importance of this study is that some properties of $MN$%
-convex functions and some related inequalities have been proven in general
terms.

\section{Main Results}

\bigskip

\begin{definition}
\label{D-1}\bigskip A function $M:(0,\infty )\times (0,\infty )\times \left[
0,1\right] \rightarrow (0,\infty )$ is called a weighted mean function if

\begin{enumerate}
\item[(WM1)] $M(u,v,\lambda )=M(v,u,1-\lambda ),$

\item[(WM2)] $M(u,u,\lambda )=u,$

\item[(WM3)] $u<M(u,v,\lambda )<v$ whenever $u<v$ and $\lambda \in \left[
0,1 \right] ,$

\item[(WM4)] $M(\alpha u,\alpha v,\lambda )=\alpha M(u,v,\lambda )$ for all $%
\alpha >0,$

\item[(WM5)] let $\lambda \in \left[ 0,1\right] $ be fixed. Then $%
M(u,v,\lambda )\leq M(w,v,\lambda )$ whenever $u\leq w$ and $M(u,v,\lambda
)\leq M(u,\omega ,\lambda )$ whenever $v\leq \omega .$

\item[(WM6)] let $u,v\in (0,\infty )$ be fixed and $u\neq v$. Then $M(u,v,.)$
is a strictly monotone and continuous function on $\left[ 0,1\right] .$

\item[(WM7)] $M\left( M(u,v,\lambda ),M(z,w,\lambda ),s\right) =M\left(
M(u,z,s),M(v,w,s),\lambda \right) $ for all $u,v,z,w\in (0,\infty )$ and $%
s,\lambda \in \left[ 0,1\right] $.

\item[(WM8)] $M(u,v,s\lambda _{1}+(1-s)\lambda _{2})=M\left( M(u,v,\lambda
_{1}),M(u,v,\lambda _{2}),s\right) $ for all $u,v\in (0,\infty )$ and $%
s,\lambda _{1},\lambda _{2}\in \left[ 0,1\right] $.
\end{enumerate}
\end{definition}

\begin{remark}
According to the above definition every weighted mean function is a mean
function with $\lambda =1/2.$ Also, By (WM6) we can say that for each $x\in %
\left[ u,v\right] \subseteq (0,\infty )$ there exists a $\lambda \in \left[
0,1\right] $ such that $x=M(u,v,\lambda )$. Morever;

i.) If $M(u,v,.)$ is a strictly increasing, then $M(u,v,0)=u$ and $%
M(u,v,1)=v $ whenever $u<v$ (i.e. $M(u,v,\lambda )$ is in the positive
direction)

ii.) If $M(u,v,.)$ is a strictly deccreasing, then $M(u,v,0)=y$ and $%
M(u,v,1)=x$ whenever $u<v$ (i.e. $M(u,v,\lambda )$ is in the negative
direction) and $M(u,v,.)(\left[ 0,1\right] )=\left[ \min \left\{ u,v\right\}
,\max \left\{ u,v\right\} \right] .$
\end{remark}

\begin{remark}
Throughout this paper, we will assume that different weighted means have the
same direction unless otherwise stated.
\end{remark}

\begin{example}
\label{E-1} 
\begin{equation*}
M(u,v,\lambda )=A(u,v,\lambda )=A_{\lambda }=(1-\lambda )u+\lambda v
\end{equation*}%
is the Weighted Arithmetic Mean, 
\begin{equation*}
M(u,v,\lambda )=G(u,v,\lambda )=G_{\lambda }=u^{1-\lambda }v^{\lambda }
\end{equation*}%
is the Weighted Geometic Mean, 
\begin{equation*}
M(u,v,\lambda )=H(u,v,\lambda )=H_{\lambda }=A^{-1}(u^{-1},v^{-1},\lambda )=%
\frac{uv}{\lambda u+(1-\lambda )v}
\end{equation*}%
is the Weighted Harmonic Mean, 
\begin{equation*}
M(u,v,\lambda )=M_{p}(u,v,\lambda )=M_{p,\lambda }=\left\{ 
\begin{array}{cc}
A^{1/p}(u^{p},v^{p},\lambda )=\left( (1-\lambda )x^{p}+\lambda y^{p}\right)
^{1/p} & p\in 
\mathbb{R}
\backslash \left\{ 0\right\} \\ 
G(u,v,\lambda )=u^{1-\lambda }v^{\lambda } & p=0%
\end{array}%
\right.
\end{equation*}%
is the $p$-Power Mean. In particular, we have the following inequality 
\begin{equation*}
M_{-1,\lambda }=H_{\lambda }\leq M_{0,\lambda }=G_{\lambda }\leq
M_{1,\lambda }=A_{\lambda }\leq M_{p,\lambda }
\end{equation*}%
for all $x,y\in (0,\infty ),t\in \left[ 0,1\right] $ and $p\geq 1.$
\end{example}

\begin{proposition}
If $M:(0,\infty )\times (0,\infty )\times \left[ 0,1\right] \rightarrow
(0,\infty )$ is a weighted mean function, then the following identities
hold: 
\begin{equation}
M\left( M\left( a,M(a,b,s),\lambda \right) ,M\left( b,M(a,b,s),\lambda
\right) ,s\right) =M(a,b,s),  \label{p-1}
\end{equation}
\begin{equation}
M\left( M(a,b,\lambda ),M(b,a,\lambda ),1/2\right) =M(a,b,1/2).  \label{p-2}
\end{equation}
\end{proposition}

\begin{proof}
If we take $v=w=M(a,b,s),\ u=a$ and $z=b$ in (WM7) and we use the property
(WM2), then we obtained the identity (\ref{p-1}). By using similar method,
if we take $u=w=a,$ $v=z=b$ and $s=1/2$\ in (WM7) and we use the properties
(WM1) and (WM2), then we obtained the identity (\ref{p-2}).
\end{proof}

\begin{definition}
Let $M$ and $N$ be two weighted means defined on the intervals $I\subseteq
(0,\infty )$ and $J\subseteq (0,\infty )$ respectively, a function $%
f:I\rightarrow J$ is called $MN$-convex (concave) if it satisfies 
\begin{equation*}
f\left( M(u,v,\lambda )\right) \leq \left( \geq \right) N\left(
f(u),f(v),\lambda \right)
\end{equation*}%
for all $u,v\in I$ and $\lambda \in \left[ 0,1\right] .$
\end{definition}

The condition (WM8) in Definition \ref{D-1} shows us that the function\ $%
M(u,v,.)$ is both $MM$-convex and $MM$-concave on $\left[ 0,1\right] $\ for
fixed $u,v\in (0,\infty ).$ It is easily seen that weighted means mentioned
in the Example \ref{E-1} hold the condition (WM8).

We note that by considering the special cases of $M$ and $N$, we obtain
several different convexity classes as $AA$-convexity (classical convexity), 
$AG$-convexity (log-convexity), $GA$-convexity, $GG$-convexity
(geometrically convexity), $HA$-convexity (harmonically convexity), $M_{p}A$%
-convexity ($p$-convexity),...,etc. For some convexity types, see (\cite%
{I13,I13d,N00,N03}).

\begin{definition}
Let $M$ and $N$ be two weighted means defined on the intervals $\left[ u,v%
\right] \subseteq (0,\infty )$ and $J\subseteq (0,\infty )$ respectively and 
$f:\left[ u,v\right] \rightarrow J$ be a function. We say that $f$ is
symmetric with respect to $M(u,v,1/2)$, if it satisfies 
\begin{equation*}
f\left( M(u,v,\lambda )\right) =f\left( M(u,v,1-\lambda )\right)
\end{equation*}%
for all $\lambda \in \left[ 0,1\right] .$
\end{definition}

\begin{theorem}
Let $M$ and $N$ be two weighted means defined on the intervals $\left[ u,v%
\right] \subseteq (0,\infty )$ and $J\subseteq (0,\infty )$ respectively. If
function $f:\left[ u,v\right] \rightarrow J$ is $MN$-convex, then the
function $f$ is bounded.
\end{theorem}

\begin{proof}
Let $K=\max \left\{ f(u),f(v)\right\} .$ For any $z=M(u,v,\lambda )$ in the
interval $\left[ u,v\right] ,$ By using $MN$-convexity of $f$ and (WM3) we
have 
\begin{equation*}
f(z)\leq N\left( f(u),f(v),\lambda \right) \leq K.
\end{equation*}%
the function $f$ is also bounded from below. For any $z\in \left( u,v\right]
,$ there exists a $\lambda _{0}\in \left( 0,1\right] $ such that $%
z=M(u,v,\lambda _{0}),$ then by using $MN$-convexity of $f$ and (\ref{p-2} )
we have 
\begin{equation}
f\left( M(u,v,1/2)\right) =f\left( M\left( z,M(v,u,\lambda _{0}),1/2\right)
\right) \leq N\left( f(z),f\left( M(v,u,\lambda _{0})\right) ,1/2\right) .
\label{p-3}
\end{equation}%
On the other hand, if $f(z)=f\left( M(v,u,\lambda _{0})\right) ,$ then $%
N\left( f(z),f\left( M(v,u,\lambda _{0})\right) ,1/2\right) =f(z)$ and thus
the function $f$ is also bounded from below.

If $f(z)\neq f\left( M(v,u,\lambda _{0})\right) ,$ then there exists $\mu
_{0}\in \left( 0,1\right) $ such that 
\begin{equation*}
N\left( f(z),f\left( M(v,u,\lambda _{0})\right) ,1/2\right) =\mu
_{0}f(z)+(1-\mu _{0})f\left( M(v,u,\lambda _{0})\right) .
\end{equation*}%
By the inequality (\ref{p-3}) and using $K$ as the upper bound, we have 
\begin{eqnarray*}
f(z) &\geq &\frac{1}{\lambda _{0}}\left[ f\left( M(u,v,1/2)\right)
-(1-\lambda _{0})f\left( M(v,u,\lambda _{0})\right) \right] \\
&\geq &\frac{1}{\lambda _{0}}\left[ f\left( M(u,v,1/2)\right) -(1-\lambda
_{0})K\right] =k.
\end{eqnarray*}%
Thus, we obtain $f(z)\geq $\ $\max \left\{ k,f(u)\right\} $ for\ any $z\in %
\left[ u,v\right] $. This completes the proof.
\end{proof}

\begin{theorem}
Let $M$ and $N$ be two weighted means defined on the intervals $I\subseteq
(0,\infty )$ and $J\subseteq (0,\infty )$ respectively. If the functions $%
f,g:I\rightarrow J$ are $MN$-convex, then $N(f(.),g(.),1/2)$ is a $MN$
-convex function.
\end{theorem}

\begin{proof}
Since $f$ and $g$ are $MN$-convex functions, we have 
\begin{equation*}
f\left( M(u,v,\lambda )\right) \leq N\left( f(u),f(v),\lambda \right)
\end{equation*}%
and 
\begin{equation*}
g\left( M(u,v,\lambda )\right) \leq N\left( g(u),g(v),\lambda \right)
\end{equation*}%
for all $x,y\in I$ and $t\in \left[ 0,1\right] .$ Then by (WM5) and (WM7) we
have 
\begin{eqnarray*}
&&N(f(.),g(.),1/2)(M(u,v,\lambda )) \\
&=&N\left( f\left( M(u,v,\lambda )\right) ,g\left( M(u,v,\lambda )\right)
,1/2\right) \\
&\leq &N\left( N\left( f(u),f(v),t\right) ,N\left( g(u),g(v),\lambda \right)
,1/2\right) \\
&=&N\left( N(f(.),g(.),1/2)(u),N(f(.),g(.),1/2)(v),\lambda \right) .
\end{eqnarray*}%
This completes the proof.
\end{proof}

We can give the following results for different convexity classes by
considering the special cases of $M$ and $N$.

\begin{corollary}
\label{c-1}Let $I,J\subseteq (0,\infty )$ and $f,g:I\rightarrow J$ .

i.) If $f$ and $g$ are convex functions, then $A(f(.),g(.),1/2)=(f+g)/2$ is
also convex function.

ii.) If $f$ and $g$ are $GA$-convex functions, then $%
A(f(.),g(.),1/2)=(f+g)/2 $ is also GA-convex function.

iii.) If $f$ and $g$ are harmonically convex functions, then $%
A(f(.),g(.),1/2)=(f+g)/2$ is also harmonically convex function.

iv.) If $f$ and $g$ are $p$-convex functions, then $A(f(.),g(.),1/2)=(f+g)/2$
is also $p$-convex function.

v.) If $f$ and $g$ are log-convex functions, then $G(f(.),g(.),1/2)=\sqrt{fg}
$ is also log-convex function.

vi.) If $f$ and $g$ are $GG$-convex functions, then $G(f(.),g(.),1/2)=\sqrt{%
fg}$ is also $GG$-convex function.

vii.) If $f$ and $g$ are $HG$-convex functions, then $G(f(.),g(.),1/2)=\sqrt{%
fg}$ is also $HG$-convex function.

viii.) If $f$ and $g$ are $AH$-convex functions, then $%
H(f(.),g(.),1/2)=2fg/(f+g)$ is also $AH$-convex function.
\end{corollary}

\begin{remark}
In Corollary \ref{c-1}, we gave results only for some convexity types. It is
possible to increase the results by considering another special cases of $M$
and $N.$
\end{remark}

\begin{theorem}
Let $M$ and $N$ be two weighted means defined on the intervals $I\subseteq
(0,\infty )$ and $J\subseteq (0,\infty )$ respectively. If $f:I\rightarrow J$
is a $MN$-convex function and $\alpha >0$, then $\alpha f$ is a $MN$-convex
function.
\end{theorem}

\begin{proof}
By using $MN$-convexity of $f$ and (WM4), we have 
\begin{equation*}
\alpha f\left( M(u,v,\lambda )\right) \leq \alpha N\left( f(u),f(v),\lambda
\right) \leq N\left( \alpha f(u),\alpha f(v),\lambda \right) .
\end{equation*}%
This completes the proof.
\end{proof}

\begin{theorem}
Let $M,N$ and $K$ be three weighted means defined on the intervals $%
I\subseteq (0,\infty ),J\subseteq (0,\infty )$ and $L\subseteq (0,\infty )$
respectively. If $f:I\rightarrow J$ is a $MN$-convex function and $%
g:J\subseteq (0,\infty )\rightarrow L$ is nondecreasing and $NK$-convex
function, then $g\circ f$ is a $MK$-convex function.
\end{theorem}

\begin{proof}
By using $MN$-convexity of $f$ , we have 
\begin{equation*}
f\left( M(u,v,\lambda )\right) \leq N\left( f(u),f(v),\lambda \right) .
\end{equation*}%
Since $g$ is $NK$-convex and nondecreasing function 
\begin{equation*}
g\left( f\left( M(u,v,\lambda )\right) \right) \leq g\left( N\left(
f(u),f(v),t\right) \right) \leq K\left( g(f(u)),g(f(v)),\lambda \right) .
\end{equation*}%
This completes the proof.
\end{proof}

\begin{theorem}
Let $M$ and $N$ be two weighted means defined on the intervals $I\subseteq
(0,\infty )$ and $J\subseteq (0,\infty )$ respectively. If the function $%
f:I\rightarrow J$ is $MN$-convex, $M\leq A$ and $N\leq A$ ($A$ is the
weighted arithmetic mean), then $f$ satisfies Lipschitz condition on any
closed interval $\left[ a,b\right] $ contained in the interior $I^{\circ }$
of $I$. Consequently, $f$ is absolutely continuous on $\left[ a,b\right] $
and continuous on $I^{\circ }.$
\end{theorem}

\begin{proof}
Choose $\varepsilon >0$ so that $a-\varepsilon $ and $b+\varepsilon $ belong
to I, and let $m_{1}$ and $m_{2}$ be the lower and upper bounds for $f$ on $%
\left[ a-\varepsilon ,b+\varepsilon \right] $ . If $u$ and $v$ are distinct
points of $\left[ a,b\right] $ and we choose a point $z$ such that 
\begin{equation*}
y=M(u,z,t),~t=\frac{\left\vert v-u\right\vert }{\varepsilon +\left\vert
v-u\right\vert },
\end{equation*}%
then 
\begin{equation*}
f(v)\leq N\left( f(u),f(z),\lambda \right) \leq A\left( f(u),f(z),\lambda
\right) =f(u)+\lambda \left[ f(z)-f(u)\right]
\end{equation*}%
\begin{equation*}
f(v)-f(u)\leq \lambda \left[ f(z)-f(u)\right] \leq \lambda (m_{2}-m_{1})<%
\frac{\left\vert v-u\right\vert }{\varepsilon }(m_{2}-m_{1})=K\left\vert
v-u\right\vert
\end{equation*}%
where $K=(m_{2}-m_{1})/\varepsilon $. Since this is true for any $u,v\in %
\left[ a,b\right] $, we conclude that $\left\vert f(v)-f(u)\right\vert \leq
K\left\vert v-u\right\vert $ as desired.

Next we recall that $f$ is absolutely continuous on $\left[ a,b\right] $ if,
corresponding to any $\varepsilon >0$, we can produce a $\delta >0$ such
that for any collection $\left\{ \left( a_{i},b_{i}\right) \right\} _{1}^{n}$
of disjoint open subintervals of $\left[ a,b\right] $ with $%
\sum_{i=1}^{n}\left( b_{i}-a_{i}\right) <\delta $, $\sum_{i=1}^{n}\left\vert
f(b_{i})-f(a_{i})\right\vert <\varepsilon $. Clearly the choice $\delta
=\varepsilon /K$ meets this requirement.

Finally the continuity of $f$ on $I^{\circ }$ is a consequence of the
arbitrariness.
\end{proof}

\begin{theorem}
Let $M$ and $N$ be two weighted means defined on the intervals $I\subseteq
(0,\infty )$ and $J\subseteq (0,\infty )$ respectively. If function $%
f_{\lambda }:I\rightarrow J$ be an arbitrary family of $MN$-convex function 
\textit{s} and let $f(u)=\sup_{\lambda }f_{\lambda }(u)$. If $K=\left\{ x\in
I:f(x)<\infty \right\} $ is nonempty, then $K$ is an interval and $f$ is $MN$%
-convex function on $K$.
\end{theorem}

\begin{proof}
Let $t\in \left[ 0,1\right] $ and $u,v\in K$ be arbitrary. Then 
\begin{eqnarray*}
&&f\left( M\left( u,v,\lambda \right) \right) \\
&=&\sup_{\alpha }f_{\alpha }\left( M\left( u,v,\lambda \right) \right) \\
&\leq &\sup_{\alpha }\left( N\left( f_{\alpha }(u),f_{\alpha }(v),\lambda
\right) \right) \\
&\leq &N\left( \sup_{\alpha }f_{\alpha }(u),\sup_{\alpha }f_{\alpha
}(v),\lambda \right) \\
&=&N\left( f(u),f(v),\lambda \right) <\infty .
\end{eqnarray*}%
This shows simultaneously that $K$ is an interval, since it contains every
point between any two of its points, and that $f$ is $MN$-convex function on 
$K$. This completes the proof of theorem.
\end{proof}

\begin{theorem}[Hermite-Hadamard's inequalities for $MN$-convex functions]
Let $M$ and $N$ be two weighted means defined on the intervals $I\subseteq
(0,\infty )$ and $J\subseteq (0,\infty )$ respectively. If function $%
f:I\rightarrow J$ is $MN$-convex, then we have 
\begin{equation}
f\left( M(u,v,1/2)\right) \leq \int_{0}^{1}N\left( f\left( M(u,v,\lambda
)\right) ,f\left( M(u,v,1-\lambda )\right) ,1/2\right) d\lambda \leq N\left(
f(u),f(v),1/2\right)  \label{2-2}
\end{equation}%
for all $u,v\in I$ with $u<v.$
\end{theorem}

\begin{proof}
Since $f:I\rightarrow 
\mathbb{R}
$ is a $MN$-convex function, by using (\ref{p-2}) we have 
\begin{eqnarray*}
f\left( M(u,v,1/2)\right) &=&f\left( M\left( M(u,v,\lambda ),M(u,v,1-\lambda
),1/2\right) \right) \\
&\leq &N\left( f\left( M(u,v,\lambda )\right) ,f\left( M(u,v,1-\lambda
)\right) ,1/2\right)
\end{eqnarray*}%
for all $u,v\in I$ and $\lambda \in \lbrack 0,1].$ Further, integrating for $%
\lambda \in \lbrack 0,1]$, we have 
\begin{equation}
f\left( M(u,v,1/2)\right) \leq \int_{0}^{1}N\left( f\left( M(u,v,\lambda
)\right) ,f\left( M(u,v,1-\lambda )\right) ,1/2\right) d\lambda .
\label{2-3}
\end{equation}%
Thus, we obtain the left-hand side of the inequality (\ref{2-2}) from (\ref%
{2-3}).

Secondly, By using $MN$-convexity of $f$ and (WM5) with (\ref{p-2}), we get 
\begin{eqnarray*}
&&N\left( f\left( M(u,v,\lambda )\right) ,f\left( M(u,v,1-\lambda )\right)
,1/2\right) \\
&\leq &N\left( N(f(u),f(v),\lambda ),N(f(u),f(v),1-\lambda ),1/2\right) \\
&=&N\left( f(u),f(v),1/2\right) .
\end{eqnarray*}%
Integrating this inequality with respect to $\lambda $ over $[0,1]$, we
obtain the right-hand side of the inequality (\ref{2-2}). This completes the
proof.
\end{proof}

We can give the following some results for different convexity classes by
considering the special cases of $M$ and $N$. It is possible to increase the
results by considering another special cases of $M$ and $N.$

\begin{corollary}
\label{c-2}Let $I,J\subseteq (0,\infty )$ and $f:I\rightarrow J$ .

i.) If $f$ is convex function, then we have the following well-known
celebrated Hermite-Hadamard's inequalities for convex functions%
\begin{eqnarray*}
f\left( A(u,v,1/2)\right) &=&f\left( \frac{u+v}{2}\right) \\
&\leq &\int_{0}^{1}A\left( f\left( A(u,v,\lambda )\right) ,f\left(
A(u,v,1-\lambda )\right) ,1/2\right) d\lambda \\
&=&\frac{1}{2(v-u)}\int_{u}^{v}f(x)+f(u+v-x)dx \\
&=&\frac{1}{v-u}\int_{u}^{v}f(x)dx \\
&\leq &A\left( f(u),f(v),1/2\right) =\frac{f(u)+f(v)}{2}.
\end{eqnarray*}

ii.) If $f$ is $GA$-convex function, then we have the following
Hermite-Hadamard's inequalities for $GA$-convex functions (see \cite[Theorem
3.1. with $s=1$]{I14}) 
\begin{eqnarray*}
f\left( G(u,v,1/2)\right) &=&f\left( \sqrt{uv}\right) \\
&\leq &\int_{0}^{1}A\left( f\left( G(u,v,\lambda )\right) ,f\left(
G(u,v,1-\lambda )\right) ,1/2\right) d\lambda \\
&=&\frac{1}{2(\ln v-\ln u)}\int_{u}^{v}f(x)+f(\frac{uv}{x})\frac{dx}{x} \\
&=&\frac{1}{\ln v-\ln u}\int_{u}^{v}\frac{f(x)}{x}dx \\
&\leq &A\left( f(u),f(v),1/2\right) =\frac{f(u)+f(v)}{2}.
\end{eqnarray*}

iii.) If $f$ is harmonically convex function, then we have the following
Hermite-Hadamard's inequalities for harmonically-convex functions (see \cite[%
2.4. Theorem]{I13d})%
\begin{eqnarray*}
f\left( H(u,v,1/2)\right) &=&f\left( \frac{2uv}{u+v}\right) \\
&\leq &\int_{0}^{1}A\left( f\left( H(u,v,\lambda )\right) ,f\left(
H(u,v,1-\lambda )\right) ,1/2\right) d\lambda \\
&=&\frac{uv}{2\left( v-u\right) }\int_{u}^{v}f(x)+f\left( \left[
u^{-1}+v^{-1}-x^{-1}\right] ^{-1}\right) \frac{dx}{x^{2}} \\
&=&\frac{uv}{v-u}\int_{u}^{v}\frac{f(x)}{x^{2}}dx \\
&\leq &A\left( f(u),f(v),1/2\right) =\frac{f(u)+f(v)}{2}.
\end{eqnarray*}

iv.) If $f$ is $p$-convex function ($p\neq 0$), then we have the following
Hermite-Hadamard's inequalities for $p$-convex functions (see \cite[Theorem 2%
]{I13})%
\begin{eqnarray*}
f\left( M_{p}(u,v,1/2)\right) &=&f\left( \left[ \frac{u^{p}+v^{p}}{2}\right]
^{1/p}\right) \\
&\leq &\int_{0}^{1}A\left( f\left( M_{p}(u,v,\lambda )\right) ,f\left(
M_{p}(u,v,1-\lambda )\right) ,1/2\right) d\lambda \\
&=&\frac{p}{2\left( v^{p}-u^{p}\right) }\int_{u}^{v}f(x)+f\left( \left[
u^{p}+v^{p}-x^{p}\right] ^{1/p}\right) \frac{dx}{x^{1-p}} \\
&=&\frac{p}{v^{p}-u^{p}}\int_{u}^{v}\frac{f(x)}{x^{1-p}}dx \\
&\leq &A\left( f(u),f(v),1/2\right) =\frac{f(u)+f(v)}{2}.
\end{eqnarray*}

v.) If $f$ is log-convex function, then we have the following
Hermite-Hadamard's inequalities for log-convex functions (see \cite[Theorem
2.1]{DM98})%
\begin{eqnarray*}
f\left( A(u,v,1/2)\right) &=&f\left( \frac{u+v}{2}\right) \\
&\leq &\int_{0}^{1}G\left( f\left( A(u,v,\lambda )\right) ,f\left(
A(u,v,1-\lambda )\right) ,1/2\right) d\lambda \\
&=&\frac{1}{v-u}\int_{u}^{v}\sqrt{f(x)f(u+v-x)}dx \\
&\leq &G\left( f(u),f(v),1/2\right) =\sqrt{f(u)f(v)}.
\end{eqnarray*}

vi.) If $f$ is $GG$-convex function, then we have the following
Hermite-Hadamard's inequalities for $GG$-convex functions (see \cite[the
inequality (7)]{I14b})%
\begin{eqnarray*}
f\left( G(u,v,1/2)\right) &=&f\left( \sqrt{uv}\right) \\
&\leq &\int_{0}^{1}G\left( f\left( G(u,v,\lambda )\right) ,f\left(
G(u,v,1-\lambda )\right) ,1/2\right) d\lambda \\
&=&\frac{1}{\ln v-\ln u}\int_{u}^{v}\sqrt{f(x)f\left( \frac{uv}{x}\right) }%
\frac{dx}{x} \\
&\leq &G\left( f(u),f(v),1/2\right) =\sqrt{f(u)f(v)}.
\end{eqnarray*}

vii.) If $f$ is $HG$-convex function, then we have%
\begin{eqnarray*}
f\left( H(u,v,1/2)\right) &=&f\left( \frac{2uv}{u+v}\right) \\
&\leq &\int_{0}^{1}G\left( f\left( H(u,v,\lambda )\right) ,f\left(
H(u,v,1-\lambda )\right) ,1/2\right) d\lambda \\
&=&\frac{uv}{v-u}\int_{u}^{v}\sqrt{f(x)f\left( \left[ u^{-1}+v^{-1}-x^{-1}%
\right] ^{-1}\right) }\frac{dx}{x^{2}} \\
&\leq &G\left( f(u),f(v),1/2\right) =\sqrt{f(u)f(v)}.
\end{eqnarray*}

viii.) If $f$ is $AH$-convex function, then we have%
\begin{eqnarray*}
f\left( A(u,v,1/2)\right) &=&f\left( \frac{u+v}{2}\right) \\
&\leq &\int_{0}^{1}H\left( f\left( A(u,v,\lambda )\right) ,f\left(
A(u,v,1-\lambda )\right) ,1/2\right) d\lambda \\
&=&\frac{2}{v-u}\int_{u}^{v}\frac{f(x)f(u+v-x)}{f(x)+f(u+v-x)}dx \\
&\leq &A\left( f(u),f(v),1/2\right) =\frac{f(u)+f(v)}{2}.
\end{eqnarray*}
\end{corollary}

\begin{theorem}
Let $M$ and $N$ be two weighted means defined on the intervals $\left[ u,v%
\right] \subseteq (0,\infty )$ and $J\subseteq (0,\infty )$ respectively. If
function $f:\left[ u,v\right] \rightarrow J$ is $MN$-convex and symmetric
with respect to $M(u,v,1/2)$, then we have 
\begin{equation}
f\left( M(u,v,1/2)\right) \leq f(x)\leq N\left( f(u),f(v),1/2\right)
\label{2-4}
\end{equation}%
for all $x\in I.$
\end{theorem}

\begin{proof}
Let $x\in \left[ u,v\right] $ be arbitrary point. Then there exists a $%
\lambda \in \left[ 0,1\right] $ such that $x=M(u,v,\lambda ).$ Since $f:%
\left[ u,v\right] \rightarrow J$ is a $MN$-convex function and symmetric
with respect to $M(u,v,1/2)$, by using (\ref{p-2}) we have 
\begin{eqnarray*}
f\left( M(u,v,1/2)\right) &=&f\left( M\left( M(u,v,\lambda ),M(u,v,1-\lambda
),1/2\right) \right) \\
&\leq &N\left( f\left( M(u,v,\lambda )\right) ,f\left( M(u,v,1-\lambda
)\right) ,1/2\right) \\
&=&f(x).
\end{eqnarray*}%
Thus, we obtain the left-hand side of the inequality (\ref{2-4}). Secondly,
By using $MN$-convexity of $f$ and (WM5) with (\ref{p-2}), we get 
\begin{eqnarray*}
&&f(x)=N\left( f\left( M(u,v,\lambda )\right) ,f\left( M(u,v,1-\lambda
\right) ,1/2\right) \\
&\leq &N\left( N(f(u),f(v),\lambda ),N(f(u),f(v),1-\lambda ),1/2\right) \\
&=&N\left( f(u),f(v),1/2\right) .
\end{eqnarray*}%
This completes the proof.
\end{proof}

We can give the following some results for different convexity classes by
considering the special cases of $M$ and $N$. It is possible to increase the
results by considering another special cases of $M$ and $N.$

\begin{corollary}
\label{c-2 copy(1)}Let $I,J\subseteq (0,\infty )$ and $f:I\rightarrow J$ .

i.) If $f$ is a convex function and symmetric with respect to $(u+v)/2$,
then we have the following inequalities for convex functions (see \cite[%
Theorem 2]{D16})%
\begin{equation*}
f\left( \frac{u+v}{2}\right) \leq f(x)\leq \frac{f(u)+f(v)}{2}.
\end{equation*}

ii.) If $f$ is a GA-convex function and symmetric with respect to $\sqrt{uv} 
$, then we have the following inequalities for convex functions (see \cite[%
Theorem 2.9]{I19})%
\begin{equation*}
f\left( \sqrt{uv}\right) \leq f(x)\leq \frac{f(u)+f(v)}{2}.
\end{equation*}

iii.) If $f$ is a $p$-convex function and symmetric with respect to $\left( 
\frac{u^{p}+v^{p}}{2}\right) ^{1/p}$, then we have the following
inequalities for convex functions (see \cite[Theorem 2.2 ]{I20})%
\begin{equation*}
f\left( \left[ \frac{u^{p}+v^{p}}{2}\right] ^{1/p}\right) \leq f(x)\leq 
\frac{f(u)+f(v)}{2}.
\end{equation*}
\end{corollary}

\section{Conclusion}

The aim of this article is to determine that a mean is called the weighted
mean when it meets what conditions, and also is to give a general definition
of $MN$-convex functions. The importance of this study is that some
properties of $MN$-convex functions and some related inequalities have been
proven in general terms via this general definition of $MN$-convex functions.

\end{document}